\documentclass{article}

\usepackage{amsmath, amssymb, amsthm, amsfonts, bm}
\usepackage{enumerate}

\numberwithin{equation}{section}
\newtheorem{theorem}{Theorem}[section]
\newtheorem{cthm}{Theorem}
\newenvironment{ctheorem}[1]
  {\renewcommand\thecthm{#1}\cthm}
  {\endcthm}

\newtheorem{ccor}{Corollary}
\newenvironment{ccorollary}[1]
  {\renewcommand\theccor{#1}\ccor}
  {\endccor}
\newtheorem{lemma}{Lemma}[section]
\theoremstyle{definition}
\newtheorem{definition}{Definition}[section]
\theoremstyle{remark}

\newcommand{\Rnum}[1]{\uppercase\expandafter{\romannumeral #1\relax}}
\newcommand{\rnum}[1]{\romannumeral #1\relax}
\newcommand{\mr}[1]{\mathrm{#1}}
\newcommand{\mb}[1]{\mathbb{#1}}
\newcommand{\mc}[1]{\mathcal{#1}}

\DeclareMathOperator{\av}{Av}
\DeclareMathOperator{\mt}{\mathcal{M}_\mathcal{T}}
\DeclareMathOperator{\md}{\mathcal{M}_d}

\title{An alternative proof of a sharp generalization of an integral inequality for the dyadic maximal operator and applications}
\author{Eleftherios N. Nikolidakis}
\date{}

\begin{document}
\maketitle
\footnotetext{ {\em E-mail address}: lefteris@math.uoc.gr}

\begin{abstract}
We give an alternative proof of a sharp generalization of an integral inequality for the dyadic maximal operator due to which the evaluation of the Bellman function of this operator with respect to two variables, is possible. This last mentioned inequality, which was first noticed in \cite{3}, also generalizes in a certain direction the results of \cite{7}.
\end{abstract}

\section{Introduction} \label{sec:1}
The dyadic maximal operator on $\mb R^n$ is a useful tool in analysis and is defined by
\begin{equation} \label{eq:1p1}
\md \phi(x) = \sup \left\{ \frac{1}{|Q|}\int_Q|\phi(y)|\,\mr dy: x\in Q,\ Q\subseteq \mb R^n\ \text{is a dyadic cube}\right\},
\end{equation}
for every $\phi\in L_\text{loc}^1(\mb R^n)$, where the dyadic cubes are those formed by the grids $2^{-N}\mb Z^n$, for $N=0, 1, 2,\ldots$.
As is well known it satisfies the following weak type (1,1) inequality
\begin{equation} \label{eq:1p2}
\left|\left\{x\in\mb R^n : \md \phi(x) > \lambda\right\}\right| \leq
\frac{1}{\lambda} \int_{\{\md\phi>\lambda\}}|\phi(y)|\,\mr dy,
\end{equation}
for every $\phi\in L^1(\mb R^n)$ and every $\lambda>0$, from which it is easy to get the following $L^p$-inequality
\begin{equation} \label{eq:1p3}
\|\md\phi\|_p \leq \frac{p}{p-1}\|\phi\|_p,
\end{equation}
for every $p>1$ and $\phi\in L^p(\mb R^n)$.

It is easy to see that the weak type inequality \eqref{eq:1p2} is best possible. It has also been proved that \eqref{eq:1p3} is best possible (see \cite{1}, \cite{2} for general martingales and \cite{15} for dyadic ones).

For the study of the dyadic maximal operator it is desirable for one to find refinements of the above mentioned inequalities. Concerning \eqref{eq:1p2}, improvements have been given in, \cite{8} and \cite{9}. If we consider \eqref{eq:1p3}, there is a refinement of it if one fixes the $L^1$-norm of $\phi$. That is we wish to find explicitly the following function (named as Bellman) of two variables $f$ and $F$.
\begin{equation} \label{eq:1p4}
B_Q^{(p)}(f,F) = \sup\left\{ \frac{1}{|Q|} \int_Q (\md \phi)^p : \phi \geq 0,\ \frac{1}{|Q|}\int_Q\phi=f,\ \frac{1}{|Q|}\int_Q\phi^p=F\right\},
\end{equation}
where $Q$ is a fixed dyadic cube and $f, F$ are such that $0< f^p\leq F$.

This function was first evaluated in \cite{4}. In fact it has been explicitly computed in a much more general setting of a non-atomic probability space $(X,\mu)$ equipped with a tree structure $\mc T$, which is similar to the structure of the dyadic subcubes of $[0,1]^n$ (see the definition in Section \ref{sec:2}). Then we define the associated maximal operator by
\begin{equation} \label{eq:1p5}
\mt\phi(x) = \sup\left\{ \frac{1}{\mu(I)}\int_I |\phi|\,\mr d\mu: x\in I\in \mc T\right\},
\end{equation}
for every $\phi\in L^1(X,\mu)$.

Moreover \eqref{eq:1p2} and \eqref{eq:1p3} still hold in this setting and remain sharp. Now if we wish to refine \eqref{eq:1p3} we should introduce the so-called Bellman function of the dyadic maximal operator of two variables given by
\begin{equation} \label{eq:1p6}
B_{\mc T}^{(p)}(f,F) = \sup\left\{ \int_X(\mt\phi)^p\,\mr d\mu: \phi\geq 0,\ \int_X\phi\,\mr d\mu = f,\ \int_X\phi^p\,\mr d\mu = F\right\},
\end{equation}
where $0< f^p\leq F$.
This function of course generalizes \eqref{eq:1p4}. In \cite{4} it is proved that
\[
B_{\mc T}^{(p)}(f,F) = F\,\omega_p\!\left(\frac{f^p}{F}\right)^p,
\]
where $\omega_p : [0,1] \to \bigl[1,\frac{p}{p-1}\bigr]$, is defined by $\omega_p(z) = H_p^{-1}(z)$, and $H_p(z)$ is given by $H_p(z) = -(p-1)z^p + pz^{p-1}$. As a consequence $B_{\mc T}^{(p)}(f,F)$ does not depend on the structure of the tree $\mc T$.
The technique for the evaluation of \eqref{eq:1p6}, that is used in \cite{4}, is based on an effective linearization of the dyadic maximal operator that holds on an adequate class of functions called $\mc T$-good (see the definition in Section 2), which is enough to describe the problem that is settled on \eqref{eq:1p6}.
In \cite{7} now a different approach has been given, for the evaluation of \eqref{eq:1p6}. This was actually done for the Bellman function of three variables in a different way, avoiding the calculus arguments that are given in \cite{4}. More precisely the following is a consequence of the results in \cite{7}.

\begin{ctheorem}{A} \label{thm:a}
Let $\phi\in L^p(X,\mu)$ be non-negative, with $\int_X\phi\,\mr d\mu=f$. Then the following inequality is true
\begin{equation} \label{eq:1p7}
\int_X(\mt\phi)^p\,\mr d\mu \leq -\frac{1}{p-1}f^p + \frac{p}{p-1}\int_X \phi\,(\mt\phi)^{p-1}\,\mr d\mu.
\end{equation}
\end{ctheorem}
This inequality, as we will see in this paper enables us to find a direct proof of the exact evaluation of \eqref{eq:1p6} (we present it for completeness-for a more general approach see \cite{7}). For this evaluation we will also need a symmetrization principle that can be found in \cite{7} and which is presented as Theorem \ref{thm:2p1} below. In this paper we will prove the following generalization of Theorem \ref{thm:a}.
\begin{ctheorem}{1} \label{thm:1}
Let $\phi: (X,\mu)\to \mb R^+$ be $\mc T$-good such that $\int_X\phi\,\mr d\mu=f$. Then for every $q\in [1,p]$ the following inequality holds
\begin{equation} \label{eq:1p8}
\int_X(\mt\phi)^p\,\mr d\mu \leq -\frac{q}{p-1}f^p + \left(\frac{p}{p-1}\right)^q\int_X\phi^q(\mt\phi)^{p-q}\,\mr d\mu.
\end{equation}
Additionally \eqref{eq:1p8} is best possible for any given $q\in[1,p]$ and $f>0$.
\end{ctheorem}
Obviously Theorem \ref{thm:1} generalizes \eqref{eq:1p7}.
We will first prove Theorem \ref{thm:1}, for the case $q=1$, that is we will provide a proof of \eqref{eq:1p7}. This can be seen in Section 3. This is done by using the linearization technique that appears in \cite{4}. By using now another technique and the statement of Theorem \ref{thm:1} it is possible for us to give a proof of the Theorem that appears just below (mentioned as Theorem 2), which generalizes Theorem 1 and which is the following.
\begin{ctheorem}{2} \label{thm:2}
Let $\phi$ be as in the hypothesis of Theorem 1 and suppose that $q\in [1,p]$. Then the following inequality is true for any $\beta>0$.
\begin{multline} \label{eq:1p9}
\begin{aligned}&\int_X(\mt\phi)^p\,\mr d\mu \leq -\frac{q(\beta+1)}{(p-1)q\beta+(p-q)}f^p + \end{aligned} \\
+\frac{p(\beta+1)^q}{(p-1)q\beta+(p-q)}\int_X\phi^q(\mt\phi)^{p-q}\,\mr d\mu.
\end{multline}
Additionally \eqref{eq:1p9} is best possible for any given $q\in[1,p]$, $f>0$ and $\beta$ such that $0<\beta \leq \frac{1}{p-1}$. By this we mean that if one fixes the second constant appearing on the right hand side of inequality \eqref{eq:1p9}, then we cannot increase the absolute value of the first constant appearing in front of $f^p$, in a way such that \eqref{eq:1p9} still holds.
\end{ctheorem}

We need also to mention that this inequality is a consequence of the results of \cite{3}. The main core of \cite{3} is the proof of a stronger inequality,
and for this one we are forced to use the linearization technique that was introduced in \cite{4}. A simple application of this last mentioned inequality
gives Theorem \ref{thm:2}, as one can see in \cite{3}, but in this paper we use the linearization technique only for the proof of Theorem \ref{thm:1}, which is now simplified. In \cite{3}, we use a refinement of this linearization technique in order to produce the stronger inequality that appears there, and for this purpose we look at this technique in more depth. We should also mention that the role of this stronger inequality is to give us a tool to approach more general Bellman functions of the dyadic maximal operator that involve more variables (and in fact integral-which is a difficult task) and for this reason we give in \cite{3} another approach, different from the one that appears here, in order to give stronger results. That is we use in \cite{3} the depth of this linearization technique among other arguments that we hope to give us certain Bellman functions of more integral variables.

The purpose of the present paper is to describe a proof of one partial result that comes immediately from the results in \cite{3}. To be more precise we first give in Section 3 a proof of Theorem A. The important in this proof is that it uses only the linearization setting of the problem, but not the ingenious arguments that appear in \cite {4}. This is not strange because by using this approach we reach to an inequality that is simpler by the one that the author in \cite{4} reaches, which provided him the way to evaluate the Bellman function of interest. But as we will see in the same Section, by using the approach of \cite{7} we can reach to the Bellman function by a different path. What we mean is that the inequality that states Theorem A, is not as innocent as it seems, even that it is produced by an elementary simple manner. But we should also mention the following. This inequality, \eqref{eq:1p7}, and only the statement of this, enables as to give a direct proof of the inequality \eqref{eq:1p9}. This last inequality gives us by a simple replacement of the exponent $q$ (the first independent variable) by $p$, the precise results as appear in \cite{4}, as we shall see at the end of this paper. This means that by using only the linearization setting of the dyadic maximal operator, we can avoid the ingenious mechanism that appears in \cite{4}, and reach to the same inequality which is
(4.25), page 326 of \cite{4}, which after a suitable minimization gives us the Bellman function.

We should also note that the above results have additinal applications in view of symmetrization principles, known for the dyadic maximal operator and is the following consequence of Theorem 2.
\begin{ccorollary}{1} \label{cor:1}
For any $g: (0,1]\to\mb R^+$ non-increasing such that $\int_0^1 g(u)\,\mr du=f$, the following inequality is true for any $\beta>0$ and sharp for any
$\beta$ such that $0<\beta \leq\frac{1}{p-1}$.
\begin{multline} \label{eq:1p10}
\begin{aligned} &\int_0^1\left(\frac{1}{t}\int_0^tg(u)\,\mr du\right)^p\mr dt \leq
-\frac{q(\beta+1)}{(p-1)q\beta+(p-q)}f^p +  \end{aligned}\\
+\frac{p(\beta+1)^q}{(p-1)q\beta+(p-q)}\int_0^1\left(\frac{1}{t}\int_0^t g(u)\,\mr du\right)^{p-q}\!\!g^q(t)\,\mr dt.
\end{multline}
\end{ccorollary}
For the case $q=1$, and the value $\beta=\frac{1}{p-1}$, this inequality is well known and is in fact equality, as can be seen by applying a simple integration by parts argument. We conclude that by using the dyadic maximal operator effectively and related symmetrization principles associated to it we can prove inequalities of Hardy type. Note also that these types of inequalities involve
parameters inside them, and the validity of them still remains true as much as their sharpness. These type of inequalities as \eqref{eq:1p9} or \eqref{eq:1p10}, generalize inequality \eqref{eq:1p7} in two important directions, and this is the appearance of the two parameters involved.

At last we mention that the evaluation of \eqref{eq:1p6} has been given by an alternative method in \cite{10}, while certain Bellman functions
corresponding to several problems in harmonic analysis, have been studied in \cite{5}, \cite{6}, \cite{7}, \cite{12}, \cite{13} and \cite{14}.

\section{Preliminaries} \label{sec:2}
Let $(X,\mu)$ be a non-atomic probability space. We give the following from \cite{4} or \cite{7}.

\begin{definition} \label{def:2p1}
{\em A set $\mc T$ of measurable subsets of $X$ will be called a tree if the following are satisfied}
\begin{enumerate}[i)]{\em
\item $X\in\mc T$ and for every $I\in\mc T$, $\mu(I) > 0$.
\item For every $I\in\mc T$ there corresponds a finite or countable subset $C(I)$ of $\mc T$ containing at least two elements such that
\vspace{-5pt}}
\begin{enumerate}[a)]{\em
\item the elements of $C(I)$ are pairwise disjoint subsets of $I$
\item $I = \bigcup\, C(I)$.}
\end{enumerate}{\em
\item $\mc T = \bigcup_{m\geq 0} \mc T_{(m)}$, where $\mc T_{(0)} = \left\{ X \right\}$ and
\[
\mc T_{(m+1)} = \bigcup_{I\in \mc T_{(m)}} C(I).
\]
\item The following holds
\[
\lim_{m\to\infty} \sup_{I\in \mc T_{(m)}} \mu(I) = 0
\]
}
\end{enumerate}
\end{definition}

For the proof of Theorem \ref{thm:1} we will need an effective linearization for the operator $\mt$ that was introduced in \cite{4}. We describe it as appears there and use it in the sequel.

For every $\phi\in L^1(X,\mu)$, non negative, and $I\in\mc T$ we define $\av_I(\phi) = \frac{1}{\mu(I)}\int_I\phi\,\mr d\mu$.
We will say that $\phi$ is $\mc T$-good if the set
\[
\mc A_\phi = \left\{x\in X: \mt \phi(x) > \av_I(\phi)\ \text{for all}\ I\in\mc T\ \text{such that}\ x\in I\right\}
\]
has $\mu$-measure zero. \\
Let now $\phi$ be $\mc T$-good and $x\in X\!\setminus\!\mc A_\phi$. We define $I_\phi(x)$ to be the largest in the nonempty set
\[
\left\{ I\in\mc T: x\in I\ \text{and}\ \mt\phi(x) = \av_I(\phi)\right\}.
\]
Now given $I\in\mc T$ let
\begin{align*}
A(\phi,I) &= \left\{x\in X\!\setminus\!\mc A_\phi: I_{\phi}(x) = I\right\}\subseteq I\ \ \text{and} \\
S_\phi &= \left\{I\in\mc T: \mu(A(\phi,I))>0\right\} \cup \left\{X\right\}.
\end{align*}
Obviously then
\[
\mt\phi = \sum_{I\in S_\phi} \av_I(\phi) \chi_{A(\phi,I)},\ \mu\text{-a.e.},
\]
where $\chi_E$ is the characteristic function of $E$.
We also define the following correspondence $I\to I^\star$ by: $I^\star$ is the smallest element of $\left\{J\in S_\phi: I\subsetneq J\right\}$. It is defined for every $I\in S_\phi$, except $X$. Also it is obvious that the $A(\phi,I)$'s are pairwise disjoint and that
\[
\mu\left(\bigcup_{I\notin S_\phi} A(\phi,I)\right) = 0,
\]
so that
\[
\bigcup_{I\in S_\phi} A(\phi,I) \approx X,
\]
where by $A\approx B$ we mean that
\[
\mu(A\!\setminus\!B) = \mu(B\!\setminus\!A) = 0.
\]
Now the following is true (see \cite{4}).
\begin{lemma} \label{lem:2p1}
Let $\phi$ be $\mc T$-good
\begin{enumerate}[i)]
\item If $I, J\in S_\phi$ then either $A(\phi,J)\cap I = \emptyset$ or $J\subseteq I$.
\item If $I\in S_\phi$ then there exists $J\in C(I)$ such that $J\notin S_\phi$.
\item For every $I\in S_\phi$ we have that $I\approx \underset{\substack{J\in S_\phi\\J\subseteq I}}{\bigcup} A(\phi,J)$.
\item For every $I\in S_\phi$ we have that
\[
A(\phi,I) = I \setminus \underset{\substack{J\in S_\phi\\ J^\star=I}}{\bigcup} J,
\]
so that
\[
\mu(A(\phi,I)) = \mu(I) - \sum_{\substack{J\in S_\phi\\ J^\star=I}} \mu(J).
\]
\end{enumerate}
\end{lemma}
\noindent From the above we see that
\[
\av_I(\phi) = \frac{1}{\mu(I)} \sum_{\substack{J\in S_\phi\\ J\subseteq I}} \int_{A(\phi,J)}\phi\,\mr d\mu.
\]
In the sequel we will also need the notion of the decreasing rearrangement of a $\mu$-measurable function defined on $X$. This is given by the following equation
\[
\phi^\star(t) = \sup_{\substack{e\subseteq X\\ \mu(e) \geq t}} \Bigl[ \inf_{x\in e}|\phi(x)| \Bigr],\ \ t\in (0,1].
\]
This is a non-increasing, left continuous function defined on $(0,1]$ and equimeasurable to $|\phi|$ (that is $\mu\!\left(\left\{|\phi|>\lambda\right\}\right) = \left|\left\{\phi^\star>\lambda\right\}\right|$, for any $\lambda>0$). A more intuitive definition of $\phi^\star$ is that it describes a rearrangement of the values of $|\phi|$ in decreasing order.
We are now ready to state the following, which appears in \cite{7} and can be viewed as a symmetrization principle for the dyadic maximal operator.
\begin{theorem} \label{thm:2p1}
The following equality is true
\begin{multline} \label{eq:2p1}
\begin{aligned} &\sup\left\{ \int_K G_1(\mt\phi)\, G_2(\phi)\,\mr d\mu: \phi^\star = g,\ \phi\geq 0,\right. \\
&\hspace{30pt} \left. \vphantom{\int_K} K\, \text{measurable subset of}\ X\ \text{with}\ \mu(K)=k\right\} = \end{aligned} \\
=\int_0^k G_1\!\!\left(\frac{1}{t}\int_0^t g\right) G_2(g(t))\,\mr dt,
\end{multline}
where $G_i: [0,+\infty) \to [0,+\infty)$ are increasing functions for $i=1,2$, while $g: (0,1]\to \mb R^+$ is non-increasing. Additionally the supremum in \eqref{eq:2p1} is attained by some $(\phi_n)$ such that $\phi_n^\star = g$, for every pair of functions $(G_1, G_2)$.
\end{theorem}
We will need the above theorem in order to complete, as is done in \cite{7}, the evaluation of the Bellman function of the dyadic maximal operator, \eqref{eq:1p6}, by using \eqref{eq:1p7}, which will be proved right below.

\section{Proof of the inequality \eqref{eq:1p7}} \label{sec:3}
We now proceed to the
\begin{proof}[Proof of Theorem \ref{thm:1} (for $q=1$)] ~ \\
Suppose that $\phi$ is $\mc T$-good, non-negative, such that $\int_X\phi\,\mr d\mu=f$. We will prove that
\[
\int_X(\mt\phi)^p\,\mr d\mu \leq -\frac{1}{p-1}f^p + \frac{p}{p-1}\int_X\phi\, (\mt\phi)^{p-1}\,\mr d\mu.
\]
We use the linearization technique mentioned in the previous Section. As we mentioned there, $\mt\phi$ can be written as
\begin{equation} \label{eq:3p1}
\mt\phi = \sum_{I\in S_\phi} \av_I(\phi) \chi_{A(\phi,I)},\ \ \mu\text{-almost everywhere on}\ X.
\end{equation}
Integrating \eqref{eq:3p1} over $X$ we see that
\begin{equation} \label{eq:3p2}
\int_X(\mt\phi)^p\,\mr d\mu = \sum_{I\in S_\phi} a_Iy_I^p,
\end{equation}
where $a_I = \mu(A(\phi,I))$ and $y_I = \av_I(\phi)$. \\
Additionally
\begin{equation} \label{eq:3p3}
\int_X\phi\,(\mt\phi)^{p-1}\,\mr d\mu = \sum_{I\in S_\phi} \int_{A(\phi,I)} \phi\,(\mt\phi)^{p-1}\,\mr d\mu = \sum\Bigl(\int_{A_I} \phi\,\mr d\mu\Bigr)y_I^{p-1},
\end{equation}
where $A_I = A(\phi,I)$ and $y_I$ are defined as above. \\
Consider now the difference
\[
\Delta = \int_X(\mt\phi)^p\,\mr d\mu - \frac{p}{p-1}\int_X\phi\,(\mt\phi)^{p-1}\,\mr d\mu
\]
which equals due to \eqref{eq:3p2} and \eqref{eq:3p3}, to
\begin{equation} \label{eq:3p4}
\Delta = \sum_{I\in S_\phi} a_Iy_I^p - \frac{p}{p-1}\sum_{I\in S_\phi}\Bigl(\int_{A_I}\phi\,\mr d\mu\Bigr)y_I^{p-1}.
\end{equation}
At this point we use the Lemma \ref{lem:2p1} \rnum 4), and conclude that
\begin{equation} \label{eq:3p5}
\int_{A_I}\phi\,\mr d\mu = \int_I\phi\,\mr d\mu - \sum_{\substack{J\in S_\phi\\J^\star=I}}\int_J\phi\,\mr d\mu = \mu(I)y_I - \sum_{\substack{J\in S_\phi\\J^\star=I}} \mu(J)y_J.
\end{equation}
Thus \eqref{eq:3p4}, in view of \eqref{eq:3p5} gives
\begin{align} \label{eq:3p6}
\Delta &= \sum_{I\in S_\phi} a_Iy_I^p - \frac{p}{p-1}\sum_{I\in S_\phi}\Biggl(\mu(I)y_I - \sum_{\substack{J\in S_\phi\\J^\star=I}}\mu(J)y_J\Biggr)y_I^{p-1} = \notag \\
 &= \sum_{I\in S_\phi}a_Iy_I^p - \frac{p}{p-1}\sum_{I\in S_\phi}\mu(I)y_I^p + \frac{p}{p-1}\sum_{I\in S_\phi}\Biggl(\sum_{\substack{J\in S_\phi\\ J^\star=I}}\mu(J)y_J\Biggr)y_I^{p-1} = \notag \\
 &= \sum_{I\in S_\phi}a_Iy_I^p - \frac{p}{p-1}\sum_{I\in S_\phi}\mu(I)y_I^p + \frac{p}{p-1}\sum_{\substack{I\in S_\phi\\ I\neq X}}\mu(I)y_I(y_{I^\star})^{p-1},
\end{align}
where in the last equation we have used the definition of the correspondence $I\to I^\star$, for $I\in S_\phi$, $I\neq X$.
We use now the elementary inequality $$p\,x\,y^{p-1} \leq x^p + (p-1)y^p,$$
which holds for every $x,y>0$, and $p>1$.
As a consequence, \eqref{eq:3p6} gives
\begin{align} \label{eq:3p7}
\Delta \leq& \sum_{I\in S_\phi}a_Iy_I^p - \frac{p}{p-1}\sum_{I\in S_\phi}\mu(I)y_I^p + \frac{1}{p-1}\sum_{\substack{I\in S_\phi\\ I\neq X}}\mu(I)\left(y_I^p + (p-1)(y_{I^\star})^p\right) = \notag \\
 & \sum_{I\in S_\phi} a_Iy_I^p - \frac{p}{p-1}\sum_{I\in S_\phi}\mu(I)y_I^p + \frac{1}{p-1}\sum_{\substack{I\in S_\phi\\ I\neq X}}\mu(I)y_I^p + \sum_{\substack{I\in S_\phi\\ I\neq X}}\mu(I)(y_{I^\star})^p.
\end{align}
We now easily see that
\begin{equation} \label{eq:3p8}
\sum_{\substack{I\in S_\phi\\ I\neq X}} \mu(I)y_I^p = \sum_{I\in S_\phi}\mu(I)y_I^p - y_X^p,
\end{equation}
and
\begin{equation} \label{eq:3p9}
\sum_{\substack{I\in S_\phi\\ I\neq X}} \mu(I)(y_{I^\star})^p = \sum_{I\in S_\phi}(\mu(I)-a_I)y_I^p,
\end{equation}
where \eqref{eq:3p9} comes from the definitions mentioned above and Lemma \ref{lem:2p1} \rnum 4). Using \eqref{eq:3p8} and \eqref{eq:3p9} in \eqref{eq:3p7} we conclude that
\begin{multline*}
\Delta \leq \sum_{I\in S_\phi}a_Iy_I^p - \frac{p}{p-1}\sum_{I\in S_\phi}\mu(I)y_I^p + \frac{1}{p-1}\sum_{I\in S_\phi}\mu(I)y_I^p - \\
- \frac{1}{p-1}y_X^p + \sum_{I\in S_\phi}(\mu(I)-a_I)y_I^p = -\frac{1}{p-1}y_X^p = -\frac{1}{p-1}f^p.
\end{multline*}
Thus we obtain the desired inequality.
\end{proof}
We now complete this section by evaluating the Bellman function of the dyadic maximal operator, \eqref{eq:1p6}, using the inequality just proved. \\
We state the following
\begin{lemma} \label{lem:3p1}
For any $\phi: (X,\mu)\to\mb R^+$, $\mc T$-good with $\int_X\phi\,\mr d\mu=f$ and $\int_X\phi^p\,\mr d\mu=F$ the following inequality is true:
\begin{equation} \label{eq:3p10}
\int_X(\mt\phi)^p\,\mr d\mu \leq F\,\omega_p\!\left(\frac{f^p}{F}\right)^p.
\end{equation}
\end{lemma}

\begin{proof}
By \eqref{eq:1p7} and H\"{o}lder's inequality we obtain
\begin{align} \label{eq:3p11}
& \begin{aligned} \Lambda_\phi &= \int_X(\mt\phi)^p\,\mr d\mu \leq -\frac{1}{p-1}f^p + \frac{p}{p-1}\int_X\phi(\mt\phi)^{p-1}\,\mr d\mu \leq \\
  &\leq -\frac{1}{p-1}f^p + \frac{p}{p-1}\left(\int_X\phi^p\,\mr d\mu\right)^\frac{1}{p}\left(\int_X(\mt\phi)^p\,\mr d\mu\right)^\frac{(p-1)}{p} = \\
  &= -\frac{1}{p-1}f^p + \frac{p}{p-1}F^\frac{1}{p}(\Lambda_\phi)^\frac{(p-1)}{p} \implies \end{aligned} \notag \\
& \frac{\Lambda_\phi}{F} \leq -\frac{1}{p-1}\frac{f^p}{F} + \frac{p}{p-1}\left(\frac{\Lambda_\phi}{F}\right)^\frac{(p-1)}{p} \implies \notag \\
& (p-1)\left(\frac{\Lambda_\phi}{F}\right) - p\left(\frac{\Lambda_\phi}{F}\right)^\frac{(p-1)}{p} \leq -\frac{f^p}{F} \implies \notag \\
& -(p-1)w^p + pw^{p-1} = H_p(w) \geq \frac{f^p}{F},
\end{align}
where $w = \left(\frac{\Lambda_\phi}{F}\right)^\frac{1}{p}$.
If $w\leq 1$ then we obviously have $\Lambda_\phi\leq F\leq F\,\omega_p\!\left(\frac{f^p}{F}\right)^p$, whereas if $w$ is such that $w>1$, we immediately see from \eqref{eq:3p11}, and the definition of $\omega_p$, that $w\leq \omega_p\!\left(\frac{f^p}{F}\right)$, or that $\Lambda_\phi\leq F\,\omega_p\left(\frac{f^p}{F}\right)^p$, that is \eqref{eq:3p10}. Our proof is now complete.
\end{proof}

We will now prove that Lemma \ref{lem:3p1} holds even if $\phi$ is not necessarily $\mc T$-good. We state it as
\begin{lemma} \label{lem:3p2}
Let $\phi\in L^p(X,\mu)$ such that $\phi\geq 0$, $\int_X\phi\,\mr d\mu = f$ and $\int_X\phi^p\,\mr d\mu=F$. Then $\int_X(\mt\phi)^p\,\mr d\mu \leq F\,\omega_p\!\left(\frac{f^p}{F}\right)^p$.
\end{lemma}

\begin{proof}
For the general nonnegative $\phi\in L^p(X,\mu)$ we consider the sequence $(\phi_m)_m$, consisting of $\mc T$-step functions, defined by
\[
\phi_m = \sum_{I\in \mc T_{(m)}} \av_I(\phi)\cdot\chi_I
\]
and then if we set
\[
\Phi_m = \sum_{I\in\mc T_{(m)}}\max\left\{\av_J(\phi): I\subseteq J\in\mc T\right\}\cdot\chi_I
\]
we easily see that $\Phi_m = \mt(\phi_m)$, since $\av_J(\phi) = \av_J(\phi_m)$, whenever $J\subseteq I\in \mc T_{(m)}$. Then it is also easy to see that
\begin{equation} \label{eq:3p12}
\int_X\phi_m\,\mr d\mu = \int_X\phi\,\mr d\mu = f,\quad F_m = \int_X\phi_m^p\,\mr d\mu \leq \int_X\phi^p\,\mr d\mu = F,
\end{equation}
for all $m\in\mb N$ and that $\Phi_m$ increases monotonically almost everywhere to $\mt(\phi)$. The relations \eqref{eq:3p12} and the fact mentioned right above can be proved easily by using the definitions of $\phi_m$ and $\Phi_m$. Since $\phi_m$ is a $\mc T$-good function (which is immediate since $\phi_m$ is a $\mc T_{(m)}$-step function) we have as a consequence of Lemma \ref{lem:3p1} that
\begin{equation} \label{eq:3p13}
\int_X\Phi_m^p\,\mr d\mu \leq F_m\cdot\omega_p\!\left(\frac{f^p}{F_m}\right)^p.
\end{equation}
We now use Lemma 2 \rnum 3) of \cite{4}, which states that the function $U(x) = \frac{\omega_p(x)^p}{x}$ is strictly decreasing on $(0,1]$. Thus since $F_m\leq F$ we must have that $F_m\,\omega_p\!\left(\frac{f^p}{F_m}\right)^p\leq F\,\omega_p\!\left(\frac{f^p}{F}\right)^p$, so letting $m\to\infty$ we get by \eqref{eq:3p13} the inequality:
\[
\int_X(\mt\phi)^p\,\mr d\mu \leq F\,\omega_p\!\left(\frac{f^p}{F}\right)^p.
\]
In this way we derive the proof of Lemma \ref{lem:3p2}.
\end{proof}
At last we prove the following
\begin{theorem} \label{thm:3p1}
The following holds
\begin{equation} \label{eq:3p14}
\sup\left\{\int_X(\mt\phi)^p\,\mr d\mu: \phi\geq 0,\ \int_X\phi\,\mr d\mu=f,\ \int_X\phi^p\,\mr d\mu=F\right\} = F\,\omega_p\!\left(\frac{f^p}{F}\right)^p,
\end{equation}
for any $f, F$ such that $0<f^p\leq F$.
\end{theorem}

\begin{proof}
Obviously by Lemma \ref{lem:3p1} and \ref{lem:3p2} we conclude that the supremum in \eqref{eq:3p14} is less or equal to the right side.
For the opposite inequality we consider the following function $g: (0,1]\to \mb R^+$ defined by
\[
g(t) = K t^{-1+\frac{1}{\alpha}},
\]
where $K$ is a fixed positive number and $\alpha>0$ will be chosen in the sequel. We search now for $K, \alpha$ such the following inequalities hold: $\int_0^1 g(t)\,\mr dt=f$ and $\int_0^1 g^p(t)\,\mr dt=F$.
In fact $\int_0^1g(t)\,\mr dt=f \iff K=\frac{f}{\alpha}$, while $\int_0^1g^p(t)\,\mr dt=F \iff \alpha = \omega_p\!\left(\frac{f^p}{F}\right)$. Indeed for these values of $K$, $\alpha$ we have that $\int_0^1g^p(t)\,\mr dt = \frac{f^p}{\alpha^p}\frac{1}{(-p+\frac{p}{\alpha}+1)} = \frac{f^p}{p\alpha^{p-1}-(p-1)\alpha^p}$, which equals to $F$ if and only if $p\alpha^{p-1}-(p-1)\alpha^p = \frac{f^p}{F}$, or equivalently when $H_p(\alpha) = \frac{f^p}{F}$, that is $\alpha = \omega_p\!\left(\frac{f^p}{F}\right)$. \\
Consider now these values of $K, \alpha$. It is immediate that for any $t\in(0,1]$, the following equality holds
\begin{equation} \label{eq:3p15}
\frac{1}{t}\int_0^tg(u)\,\mr du = \alpha g(t) = \omega_p\!\left(\frac{f^p}{F}\right)g(t).
\end{equation}
Then we use Theorem \ref{thm:2p1} in the form
\begin{equation} \label{eq:3p16}
\sup\left\{ \int_X(\mt\phi)^p\,\mr d\mu : \phi^\star = g \right\} = \int_0^1 \Bigl(\frac{1}{t}\int_0^tg\Big)^p\mr dt.
\end{equation}
By \eqref{eq:3p15}, \eqref{eq:3p16} and the integral conditions for $g$ we thus have that
\begin{equation} \label{eq:3p17}
\sup\left\{ \int_X(\mt\phi)^p\,\mr d\mu : \phi^\star=g \right\} = F\,\omega_p\!\left(\frac{f^p}{F}\right)^p.
\end{equation}
This gives us immediately, because for any $\phi$ such that $\phi^\star=g$ we have $\int_X\phi\,\mr d\mu=f$ and $\int_X\phi^p\,\mr d\mu=F$ ($\phi$ is equimeasurable to $g$), that \eqref{eq:3p14} is true. The proof of the evaluation of \eqref{eq:1p6} is now complete.
\end{proof}

\section{Proof of the inequality \eqref{eq:1p9}} \label{sec:4}

{\em Second proof of Theorem 2 (different from the one that appears in \cite{3})}
\begin{proof}
Our aim is to prove the following inequality
\begin{multline} \label{eq:4p1}
\begin{aligned}&\int_X(\mt\phi)^p\,\mr d\mu \leq -\frac{q(\beta+1)}{(p-1)q\beta+(p-q)}f^p + \end{aligned} \\
+\frac{p(\beta+1)^q}{(p-1)q\beta+(p-q)}\int_X\phi^q(\mt\phi)^{p-q}\,\mr d\mu,
\end{multline}
or equivalently the following
\begin{multline} \label{eq:4p2}
\int_X\phi^q(\mt\phi)^{p-q}\,\mr d\mu\geq A(p,q,\beta)\int_X(\mt\phi)^p\,\mr d\mu+\frac{q}{p}\frac{1}{(\beta+1)^{q-1}}f^p,
\end{multline}
where $A=A(p,q,\beta)$, is defined by $A=\frac{(q-1)\beta}{(\beta+1)^q}+\frac{p-q}{p}\frac{1}{(\beta+1)^{q-1}}$, and this will be done directly by using Theorem A. For this reason we consider the difference $$L(p,q,\beta)=A\int_X(\mt\phi)^p\,\mr d\mu-\int_X\phi^q(\mt\phi)^{p-q}\,\mr d\mu.$$
We just need to prove that $$L(p,q,\beta)\leq-\frac{q}{p}\frac{1}{(\beta+1)^{q-1}}f^p.$$
We now use the following notation for the integrals below:
$J_0=\int_X(\mt\phi)^p\,\mr d\mu$, $J_q=\int_X\phi^q(\mt\phi)^{p-q}\,\mr d\mu$ and
$J_1=\int_X\phi(\mt\phi)^{p-1}\,\mr d\mu$ which are defined for any fixed $q\in[1,p]$, and every $\phi\in L^p(X,\mu)$, which is a $\mc T$-good function, with $\int_X\phi\,\mr d\mu=f$, where $f$ is a fixed positive constant.

It is immediate then, according to Theorem A, that the inequality
\begin{equation} \label{eq:4p3}
J_0 \leq \frac{p}{p-1}J_1-\frac{1}{p-1}f^p
\end{equation}
is true. Additionally the inequality that follows should be true in view of H\"{o}lder's inequality
$$J_1 \leq J^{1/q}_q J^{(q-1)/q}_0,$$ which in turn gives
\begin{equation} \label{eq:4p4}
J_q\geq\frac{J^q_1}{J^{q-1}_0}.
\end{equation}
Now in view of the above definitions, and because of \eqref{eq:4p3}, we have that
\begin{equation} \label{eq:4p5}
L(p,q,\beta)=AJ_0-J_q\leq AJ_0-J_0^{-q+1}J_1^q.
\end{equation}
The left side of \eqref{eq:4p5} is less or equal than
$$AJ_0-J_0^{-q+1}\big(\frac{p-1}{p}J_0+\frac{1}{p}f^p\big)^q$$
in view of \eqref{eq:4p3}.
We define now the following function of the variable $x>0$,
$$G(x)=Ax-x^{-q+1}\big(\frac{p-1}{p}x+\frac{1}{p}f^p\big)^q$$
which obviously equals to
$$G(x)=Ax-x\big(\frac{p-1}{p}+\frac{1}{p}\frac{f^p}{x}\big)^q.$$
Then one can easily see that
$$\frac{d}{dx}G(x)=A-t^q+\frac{q}{p}t^{q-1}(pt-(p-1)),$$
where $t$ is defined by $t=t(x)=\frac{p-1}{p}+\frac{1}{p}\frac{f^p}{x}>\frac{p-1}{p}$. Thus $$\frac{d}{dx}G(x)=F(t):=A+(q-1)t^q-\frac{q(p-1)}{p}t^{q-1}.$$
Then we immediately see that $\frac{d}{dt}F(t)=q(q-1)t^{q-2}\big(t-\frac{p-1}{p}\big)>0$, for every $t>\frac{p-1}{p}$.
Thus we must have that $F(t)>F(\frac{p-1}{p})=A(p,q,\beta)-(\frac{p-1}{p})^q$. We define now the following function of the variable $\beta>0$, by
$h(\beta)=A(p,q,\beta)$. Then it is not difficult to see that $\frac{d}{d\beta}h(\beta)=\frac{q(q-1)}{p}\frac{1-(p-1)\beta}{(\beta+1)^{q-1}}$.
Thus $h(\beta)$ attains its maximum value at $\beta_0=\frac{1}{p-1}$, and this equals to $h(\beta_0)=(\frac{p-1}{p})^q$.
Thus $F(\frac{p-1}{p})\leq 0$. We now set $t_0=\frac{p-1}{p}$. Then by the evaluation of the derivative of $F(t)$, we see that for any $t>t_0$, we have that
$F(t)>F(t_0)$. Additionally $F$ is strictly increasing on $[\frac{p-1}{p}, +\infty)$, and $F(t)$ tends to $+\infty$ as $t$ does. Thus for any fixed $\beta>0$,
there exists a unique $t_{\beta}>t_0$, for which $F(t_{\beta})=0$. For any $\beta>0$, we define $x_{\beta}>0$, by the following relation
$$t_{\beta}=\frac{p-1}{p}+\frac{1}{p}\frac{f^p}{x_{\beta}}.$$

Then according to the facts that are given above, we have that
$\frac{d}{dx}G(x)>0$, for $x\in (0,x_{\beta})$, while $\frac{d}{dx}G(x)<0$ for $x>x_{\beta}$. Additionally the explicit expression of $x_{\beta}$ is as follows
$x_{\beta}=\frac{f^p}{pt_{\beta}-(p-1)}$. By the monotonicity properties and the definition of $G(x)$, we conclude that $L(p,q,\beta)\leq G(x_{\beta})$. On the
other hand $G(x_{\beta})=Ax_{\beta}-x_{\beta}(t_{\beta})^q=\frac{A-(t_{\beta})^q}{pt_{\beta}-(p-1)}f^p$, and we conclude our result by showing that
$\frac{A-(t_{\beta})^q}{pt_{\beta}-(p-1)}=-\frac{q}{p}\frac{1}{(\beta+1)^{q-1}}$. We restrict ourselves first on the range
$\beta\in (0,\frac{1}{p-1})$. Then
$\frac{1}{\beta+1}>t_0=\frac{p-1}{p}$, an as one can easily see after doing some simple calculations, that the following equality is true
$F(\frac{1}{\beta+1})=0$, thus we must have by the definition of $t_{\beta}$ and the monotonicity of $F(t)$, that $t_{\beta}=\frac{1}{\beta+1}$, for every
$\beta\in (0,\frac{1}{p-1})$.

Our wish is to prove $\frac{A-(t_{\beta})^q}{pt_{\beta}-(p-1)}=-\frac{q}{p}\frac{1}{(\beta+1)^{q-1}}$, or equivalently, by replacing $\frac{1}{\beta+1}$ by
$t_{\beta}$ in the preceding equation we see that is enough to show that
$$\frac{A-(t_{\beta})^q}{pt_{\beta}-(p-1)}=-\frac{q}{p}t_{\beta}^{q-1},$$
which is the same as
$$ A+(q-1)t_{\beta}^q-\frac{q}{p}(p-1)t_{\beta}^{q-1}=0,$$
or as $F(t_{\beta})=0$, which is obviously true by the definition of $t_{\beta}$.

Now for the range $\beta\in(\frac{1}{p-1}, +\infty)$, we still have
$$G(x_{\beta})=\frac{A-(t_{\beta})^q}{pt_{\beta}-(p-1)}f^p,$$
and we want to show that the right side of this equality is less or equal than $-\frac{q}{p}\frac{1}{(\beta+1)^{q-1}}$. But since $F(t_{\beta})=0$, we still
have (see also above) that the following equality is true
$$\frac{A-(t_{\beta})^q}{pt_{\beta}-(p-1)}=-\frac{q}{p}t_{\beta}^{q-1}.$$
Now since $\beta>\frac{1}{p-1}$, we obtain the immediate inequalities $t_{\beta}>t_0=\frac{p-1}{p}>\frac{1}{\beta+1}$, so from the equality right above
we conclude the desired inequality.
\end{proof}

At this point we give the following.
\begin{proof}[Proof of Corollary \ref{cor:1}]~\\
Let $g: (0,1]\to \mb R^+$ be non-increasing, such that $\int_0^1g(u)\,\mr du=f$. Fix a non-atomic probability space $(X,\mu)$ equipped with a tree structure $\mc T$, for which the $\mc T$-step functions (which are included in the $\mc T$-good functions) are dense in $L^p(X,\mu)$. Then
\eqref{eq:1p9} is true for every $L^p$-function $\phi$, as can be easily seen by arguments similar to those in Lemma 3.2. Applying Theorem \ref{thm:2p1} for the pair of functions
\[
\left(G_1(t) = t^p, G_2(t)=1\right)\quad \text{and}\quad \left(G_1'(t) = t^{p-q}, G_2'(t)=t^q\right)
\]
we conclude that there exists $\phi_n: (X,\mu)\to\mb R^+$ such that $\phi_n^\star=g$, for which
\begin{equation} \label{eq:4p6}
\lim_n\int_X(\mt\phi_n)^p\,\mr d\mu = \int_0^1\left(\frac{1}{t}\int_0^tg\right)^p\mr dt,
\end{equation}
and
\begin{equation} \label{eq:4p7}
\lim_n\int_X\phi_n^q(\mt\phi_n)^{p-q}\,\mr d\mu = \int_0^1\left(\frac{1}{t}\int_0^tg\right)^{p-q}g^q(t)\,\mr dt.
\end{equation}
Applying \eqref{eq:1p8} for every $(\phi_n)$, and taking the limits as $n\to\infty$, we conclude by \eqref{eq:4p6} and \eqref{eq:4p7} the statement of Corollary \ref{cor:1}.

We now prove that \eqref{eq:1p10} is best possible. We proceed to this as follows: We first treat the case where $\beta=\frac{1}{p-1}$. We consider the following continuous, decreasing function $g_\alpha(t) = c\, t^{-\alpha}$, defined in $(0,1]$, where $c = f(1-\alpha)$, and
$\alpha\in \bigl(0,\frac{1}{p}\bigr)$.
Then it is easy to show that $\int_0^1 g_\alpha(u)\,\mr du=f$, while $g_\alpha\in L^p(0,1)$. \\
Note that for any $t\in (0,1]$ the following equality holds $\frac{1}{t}\int_0^tg(u)\,\mr du = (\frac{p}{p-1})g(t)$. So we consider the difference
\[
J = \int_0^1\left(\frac{1}{t}\int_0^tg_\alpha\right)^p\mr dt - \left(\frac{p}{p-1}\right)^q\int_0^1g_\alpha^q(t)\left(\frac{1}{t}\int_0^tg_\alpha\right)^{p-q}\mr dt
\]
which is equal to
\[
J = \left(\frac{1}{1-\alpha}\right)^p\int_0^1g_\alpha^p(t)\,\mr dt - \left(\frac{p}{p-1}\right)^q\left(\frac{1}{1-\alpha}\right)^{p-q}\int_0^1g_\alpha^p(t)\,\mr dt.
\]
Since $\int_0^1g_\alpha^p(t)\,\mr dt = f^p(1-\alpha)^p\frac{1}{1-\alpha p}$, we have by the above evaluation of $J$, that
\begin{align*}
J &= \frac{f^p}{1-\alpha p} - \left(\frac{p}{p-1}\right)^q(1-\alpha)^q\frac{f^p}{1-\alpha p} = \\
&=-\frac{f^p}{1-\alpha p}\left[\left(\frac{p}{p-1}\right)^q(1-\alpha)^q-1\right] = -f^p \,G(\alpha),
\end{align*}
where $G(\alpha)$ is defined for any $\alpha\in \bigl(0,\frac{1}{p}\bigr)$ by $G(\alpha) = \frac{\left(\frac{p}{p-1}\right)^q(1-\alpha)^q-1}{1-\alpha p}$.
But as it is easily seen, by using de L' Hospital's rule,
\[
\lim_{\alpha\to 1/p^-}G(\alpha) = -q\left(1-\frac{1}{p}\right)^{q-1}\left(\frac{p}{p-1}\right)^q\left(-\frac{1}{p}\right) = \frac{q}{p-1}.
\]
We now prove the sharpness of \eqref{eq:1p10}, for any $\beta$ such that $0<\beta<\frac{1}{p-1}$. We fix such a $\beta$, and we consider the following continuous, decreasing function $g_\beta(t) = c\, t^{-\alpha}$, defined in $(0,1]$, where $c = f(1-\alpha)$, and $\alpha=\frac{\beta}{\beta+1}$. Then
$\alpha\in \bigl(0,\frac{1}{p}\bigr)$,
and it is easy to see that $\int_0^1 g_\beta(u)\,\mr du=f$, while for any $\beta$ as above, $g_\beta\in L^p(0,1)$. \\
Moreover $\int_0^1g^p_\beta(u)\,\mr du=\frac{f^p}{(\beta+1)^p} \frac{\beta+1}{1-\beta(p-1)}$. Note that for any $t\in (0,1]$ the following equality holds $\frac{1}{t}\int_0^tg_\beta(u)\,\mr du = (\beta+1)g_\beta(t)$. We then consider the difference
\begin{multline}
\begin{aligned}&J =\int_0^1g_\beta^q(t)\left(\frac{1}{t}\int_0^tg_\beta\right)^{p-q}\mr dt- \end{aligned} \\
-\left[\frac{(q-1)\beta}{(\beta+1)^q}+\frac{p-q}{p}\left(\frac{1}{\beta+1}\right)^{q-1}\right]\int_0^1\left(\frac{1}{t}\int_0^tg_\beta\right)^p\mr dt
\end{multline}
Then due to the above mentioned relations, we can see easily after some simple calculations that
$J=\frac{q}{p}\frac{1}{(\beta+1)^{q-1}}f^p$. The proof of Corollary 1 is now complete.
\end{proof}

Now for the proof of Theorem \ref{thm:2}, we need to prove the sharpness of \eqref{eq:1p9}. This is easy now to show, since by Theorem \ref{thm:2p1} for any $g: (0,1]\to \mb R^+$ non increasing, there exists a sequence $\phi_n: (X,\mu)\to\mb R^+$ of rearrangements of $g$ such that
\begin{equation} \label{eq:4p8}
\lim_n\int_X(\mt\phi_n)^p\,\mr d\mu = \int_0^1\left(\frac{1}{t}\int_0^tg\right)^p\mr dt
\end{equation}
and
\begin{equation} \label{eq:4p9}
\lim_n\int_X\phi_n^q (\mt\phi_n)^{p-q}\,\mr d\mu = \int_0^1g^q(t)\left(\frac{1}{t}\int_0^tg\right)^{p-q}\mr dt.
\end{equation}
We discuss now the case where $0<\beta<\frac{1}{p-1}$, and we consider the function $g_\beta$ (denoted now as $g$), constructed in the proof of Corollary \ref{cor:1}. We choose a rearrangement $\phi_{n}$ of $g$ such that
\[
\left|\int_0^1\left(\frac{1}{t}\int_0^tg\right)^p\mr dt - \int_X\left(\mt\phi_{n}\right)^p\mr d\mu\right| \leq \frac{1}{n}
\]
and
\[
\left|\int_0^1g^q(t)\left(\frac{1}{t}\int_0^tg\right)^{p-q}\mr dt - \int_X\phi_{n}^q\left(\mt\phi_{n}\right)^{p-q}\mr d\mu\right| \leq \frac{1}{n}
\]
Then, by the choice of $g$, we conclude that\eqref{eq:1p9} is best possible.
The case $\beta=\frac{1}{p-1}$ is entirely similar, so we omit it.
The proof of Theorem \ref{thm:1}, is now complete.

At last we add in this section the following note

{\bf Remark 4.1} {\em Inequality \eqref{eq:1p9} is true even in the case where $\phi$ is not $\mc T$-good, and this can be proved by the method of the proof
of Lemma 3.2, as can be easily seen.}

We state at last the following
\begin{ccorollary}{2} \label{cor:2}
Let $\phi\in L^p(X,\mu)$ be non-negative, such that $\int_X\phi\,\mr d\mu=f$ and $\int_X\phi^p\,\mr d\mu=F$, where $f, F$ are fixed variables satisfying $0<f^p \leq F$. Then the following inequality is true for every value of the parameter $\beta$
$$ \int_X(\mt\phi)^p\,\mr d\mu \leq \frac{\beta+1}{\beta}\frac{(\beta+1)^{p-1}F-f^p}{p-1}$$
\end{ccorollary}

This Corollary is an immediate consequence of Theorem 2, by setting in the inequality \eqref{eq:1p9} the value $p$ in place of $q$. This is exactly inequality
(4.25) of \cite{4}, which gives us all the information we need for the evaluation of the Bellman function \eqref{eq:1p6}, as is mentioned in the introduction.
The Proofs are now complete.

Nikolidakis Eleftherios, Visiting Professor, University of Ioannina, Department of Mathematics, Ioannina, Greece.

\end{document}